\documentclass[a4paper,10pt]{article}
\usepackage[utf8]{inputenc}
\usepackage[dvipsnames,usenames]{color}
\usepackage{amsthm,amsmath}
\usepackage{enumerate,enumitem}
\usepackage{graphicx}
\usepackage{authblk}
\usepackage{cite}

\theoremstyle{plain}
\newtheorem{theorem}{Theorem}

\newtheorem{lemma}[theorem]{Lemma}

\newtheorem*{conjecture*}{Conjecture}

\newtheorem*{question*}{Question}

\DeclareGraphicsExtensions{.pdf,.png,.eps}
\graphicspath{{figs/}}

\title{The interval number of a planar graph is at most three}

\author[1]{Guillaume Gu\'egan}
\author[2,3]{Kolja Knauer}%
\author[4]{Jonathan Rollin}%
\author[5]{Torsten Ueckerdt}%

\affil[1]{Universit\'e de Montpellier, CNRS, LIRMM, France}
\affil[2]{Aix Marseille Univ, Univ de Toulon, CNRS, LIS, Marseille, France}
\affil[3]{Departament de Matemàtiques i Informàtica,
Universitat de Barcelona (UB), Barcelona, Spain}
\affil[4]{FernUniversität in Hagen, Hagen, Germany}
\affil[5]{Karlsruhe Institute of Technology, Karlsruhe, Germany}

\begin{document}

\maketitle

\begin{abstract}
 The interval number of a graph $G$ is the minimum $k$ such that one can assign to each vertex of $G$ a union of $k$ intervals on the real line, such that $G$ is the intersection graph of these sets, i.e., two vertices are adjacent in $G$ if and only if the corresponding sets of intervals have non-empty intersection.  

 In 1983 Scheinerman and West [The interval number of a planar graph: Three intervals suffice. \textit{J.~Comb.~Theory, Ser.~B}, 35:224--239, 1983] proved that the interval number of any planar graph is at most $3$.
 However the original proof has a flaw.
 We give a different and shorter proof of this result.
\end{abstract}

\section{Introduction}

For a positive integer $k$, a \emph{$k$-interval representation} of a graph $G = (V,E)$ is a set $\{f(v) \mid v \in V\}$ where $f(v)$ is the union of at most $k$ intervals on the real line representing vertex $v\in V$, such that $uv$ ($u \neq v$) is an edge of $G$ if and only if $f(u)\cap f(v) \neq \emptyset$.
In 1979 Trotter and Harary~\cite{Tro-79} introduced the \emph{interval number} of $G$, denoted by $i(G)$, as the smallest $k$ such that $G$ has a $k$-interval representation.
In 1983 Scheinerman and West~\cite{Sch-83} constructed planar graphs with interval number at least~$3$ and proposed a proof for the following:

\begin{theorem}[Scheinerman, West~\cite{Sch-83}]\label{thm:planar-3-intervals}
 If $G$ is planar, then $i(G) \leq 3$. 
\end{theorem}

In this paper we point out a flaw in that proof and give an alternative proof of Theorem~\ref{thm:planar-3-intervals}.
The original proof is based on induction, maintaining a very comprehensive description of the (partial) representation and its properties.
However in certain cases, one of the required properties can not be maintained during the construction.
A precise description of the problematic case is given in Section~\ref{sec:old}.
Let us also remark that in his thesis~\cite{Sch-84}, Scheinerman proposes a slightly different argumentation for the considered case which also leads to the same problem.

In the remaining part of the introduction we discuss some related work.
In Section~\ref{sec:prel} we introduce the notions used in the new proof, which is then given in Section~\ref{sec:new}.
We conclude with final remarks in Section~\ref{sec:conclusion}.

\paragraph{Related work.}

A graph with $i(G)\leq 1$ is called an \emph{interval graph}.
A concept closely related to the interval number is the \emph{track number} of $G$, denoted by $t(G)$, which is the smallest $k$ such that $G$ is the union of $k$ interval graphs.
Equivalently, $t(G)$ is the smallest $k$ such that $G$ admits a $k$-interval representation that is the union of $k$ $1$-interval representations, each on a different copy of the real line (called a track) and each containing one interval per vertex.
More recently, Knauer and Ueckerdt~\cite{Kna-16} defined the \emph{local track number} of $G$, denoted by $t_\ell(G)$, to be the smallest $k$ such that $G$ is the union of $d$ interval graphs, for some $d$, where every vertex of $G$ is contained in at most $k$ of them.
It is easy to see that for every graph $G$ we have $i(G) \leq t_\ell(G) \leq t(G)$.

Gon\c{c}alves~\cite{Gon-07} proved that the track number of planar graphs is at most~$4$, which is best-possible~\cite{Gon-09}.
It is an open problem whether there is a planar graph with local track number~$4$~\cite{Kna-16}.
Balogh~\textit{et al.}~\cite{Bal-04} show that any planar graph of maximum degree at most~$4$ has interval number at most~$2$.
For further recent results about interval numbers of other classes of planar graphs and general graphs, let us refer to~\cite{Bal-99} and~\cite{deQ-16}, respectively.

Finally, let us mention that West and Shmoys~\cite{Wes-84} proved that deciding $i(G) \leq k$ is NP-complete for each $k \geq 2$, while Jiang~\cite{Jia-13} proved the same for $t(G) \leq k$, and Stumpf~\cite{Stu-15} for $t_\ell(G) \leq k$.

%

\section{Preliminaries}\label{sec:prel}

All graphs considered here are finite, simple, undirected, and non-empty.
If $G = (V,E)$ is a graph and $\{f(v) \mid v \in V\}$ is a $k$-interval representation of $G$ for some $k$, we say that a subset $S$ of the real line is \emph{intersected} by some $f(v)$ if $S \cap f(v) \neq \emptyset$.
For a point $p$ with $p \in f(v)$, we also say that $p$ is \emph{covered} by $f(v)$.
A vertex $v$ has a \emph{broken end} $b$ if $b$ is an endpoint of an interval in $f(v)$ and $b$ is not covered by any other $f(w)$ for $w\in V - \{v\}$.
We say that a vertex $v$ is \emph{displayed} if $f(v)$ contains a \emph{portion} (i.e., a non-empty open interval) not intersected by any other $f(w)$ for $w\in V - \{v\}$.
An edge $uv\in E$ is \emph{displayed} if $f(u)\cap f(v)$ contains a portion not intersected by any other $f(w)$ for $w \in V - \{u,v\}$.

The \emph{depth} of the representation $\{f(v) \mid v \in V\}$ is the largest integer $d$ such that some point of the real line is covered by at least $d$ sets $f(v)$ with $v \in V$.
For $d \geq 1$, the \emph{depth~$d$ interval number} of $G$, denoted by $i_d(G)$, is the smallest $k$ such that $G$ admits a $k$-interval representation of depth~$d$.
Scheinerman and West~\cite{Sch-83} proved that there are planar graphs $G$ with $i_2(G) \geq 4$, while Gon\c{c}alves~\cite{Gon-07} proved that $i_2(G) \leq 4$ for all planar graphs $G$.
In Theorem~\ref{thm:4-connected-case} we obtain that for any $4$-connected planar graph $G$ it holds $i_2(G) \leq 3$. Finally, we shall prove here that $i_3(G) \leq 3$ for all planar graphs $G$, which is also the original claim of Scheinerman and West.

\section{An alternative proof of Theorem~\ref{thm:planar-3-intervals}}\label{sec:new}

In this section we give a new proof for Theorem~\ref{thm:planar-3-intervals}, i.e., that every planar graph $G$ has interval number at most~$3$.
A \emph{triangulation} is a plane embedded graph in which every face is bounded by a triangle.
In other words, a triangulation is a maximally planar graph on at least three vertices with a fixed plane embedding.
As every planar graph is an induced subgraph of some triangulation (For example, iteratively adding a new vertex to a non-triangular face together with one edge to each of the face's incident vertices, eventually results in such a triangulation.) and the interval number is monotone under taking induced subgraphs, we may assume without loss of generality that $G$ is a triangulation.

A triangle in $G$ is \emph{non-empty} if its interior contains at least one vertex of $G$.
We shall construct a $3$-interval representation of $G$ by recursively splitting $G$ along its non-empty triangles.
This leaves us with the task to represent $4$-connected triangulations, i.e., triangulations whose only non-empty triangle is the outer triangle, and to ``glue'' those representation along the non-empty triangles of $G$.
More precisely, we shall roughly proceed as follows:
\begin{enumerate}[label = \textbf{(\Roman*)}]
 \item Consider a non-empty triangle $\Delta$ with inclusion-minimal interior and the set $X$ of all vertices in its interior.
 
 \item Call induction on the graph $G_1 = G - X$, obtaining a $3$-interval representation $f_1$ of $G_1$ with additional properties on how inner faces are represented.
 
 \item Since the subgraph $G_2$ of $G$ induced by $V(\Delta) \cup X$ is $4$-connected, we can utilize a recent result of the second and fourth author to decompose $G_2$ into a path and two forests.\label{enum:step-decompose}
 
 \item Using this decomposition, we define a $3$-interval representation $f_2$ of $G_2$ that coincides with $f_1$ on $V(\Delta)$.\label{enum:step-glue}
\end{enumerate}

We remark that the essentials of the construction in step~\ref{enum:step-glue} can be already found in~\cite{Axe-13}.
Also, the decomposition of a triangulation along its non-empty triangles is a common method in the field of intersection graphs; see e.g.,~\cite{Tho-86,Cha-09,Cha-12}.
Hence the key to our new proof is the most recent~\cite{Kna-17} decomposition used in step~\ref{enum:step-decompose}, which we shall state next.

Let $G = (V,E)$ be a $4$-connected triangulation with outer vertices $x,y,z$.
We denote by $u_x$ the unique inner vertex of $G$ adjacent to $y$ and $z$ (if there were several, $G$ would not be $4$-connected), and call the vertex $u_x$ the vertex \emph{opposing $x$}.
Similarly, the vertex $u_y$ opposing $y$ and the vertex $u_z$ opposing $z$ are defined.
Note that since $G$ is $4$-connected we have that $u_x,u_y,u_z$ all coincide if $|V| = 4$, and are pairwise distinct if $|V| \geq 5$.
We use the following recent result of the second and fourth author (Lemma 3.1 in~\cite{Kna-17}), which we have adapted here to the case of $4$-connected triangulations (see Figure~\ref{fig:new-proof} for an illustration):

\begin{lemma}[Knauer, Ueckerdt~\cite{Kna-17}]\label{lem:inner-decomposition}
 Let $G$ be a plane $4$-connected triangulation with outer triangle $\Delta_{\rm out} = x,y,z$ and corresponding opposing vertices $u_x,u_y,u_z$.
 Then the inner edges of $G$ can be partitioned into three forests $F_x,F_y,F_z$ such that
 \begin{itemize}
  \item $F_x$ is a Hamiltonian path of $G \setminus\{y,z\}$ going from $x$ to $u_x$,
  
  \item $F_y$ is a spanning tree of $G \setminus\{x,z\}$,
  
  \item $F_z$ is a spanning forest of $G \setminus\{y\}$ consisting of two trees, one containing $x$ and one containing $z$, unless ${G}\cong K_4$. In this case $F_z=zu_z$.
 \end{itemize}
\end{lemma}

\begin{figure}[tb]
 \centering
 \includegraphics{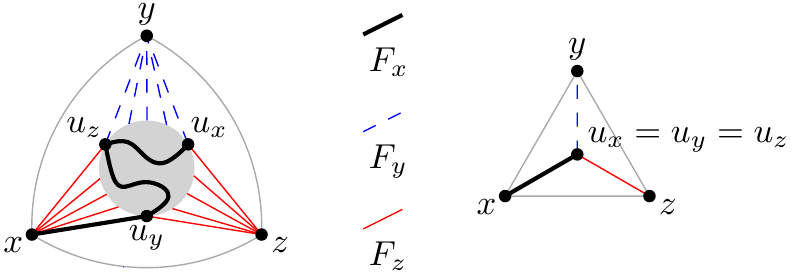}
 \caption{The decomposition in Lemma~\ref{lem:inner-decomposition} for the case $G \not\cong K_4$ (left) and $G \cong K_4$ (right).}
 \label{fig:new-proof}
\end{figure}

Note that the conditions on the decomposition in Lemma~\ref{lem:inner-decomposition} imply that the edge $xu_y$ is in $F_x$ and the edge $zu_y$ is in $F_z$ (since $x$ and $z$ are in different components of $F_z$ and $z$ is not in $F_x$).
We are now ready to give our new proof of Theorem~\ref{thm:planar-3-intervals}.

\begin{proof}[Proof of Theorem~\ref{thm:planar-3-intervals}]
 We have to show that every planar graph $G$ admits a $3$-interval representation of depth at most~$3$, in particular that $i(G) \leq 3$.
 As every planar graph is an induced subgraph of some planar triangulation and the interval number is monotone under taking induced subgraphs, we may assume without loss of generality that $G$ is a triangulation.
 
 We proceed by induction on the number $n$ of vertices in $G$, showing that $G$ admits a $3$-interval representation of depth~$3$ with the additional invariant that \textbf{(I1)} every vertex is displayed, and \textbf{(I2)} every inner face contains at least one displayed edge.
 
 The base case is $n = 3$, i.e., $G$ is a triangle with vertices $x,y,z$.
 In this case, it is easy to define a $3$-interval representation of $G$ with invariants \textbf{(I1)} and \textbf{(I2)}.
 For example, take $f(x) := [0,3]$, $f(y) := [1,4] \cup [6,7]$, and $f(z) := [2,5]$, and note that edges $xy$ and $yz$ are displayed.
 
 \medskip
 
 Now assume that $n \geq 4$, i.e., $G$ contains at least one non-empty triangle.
 Let $\Delta$ be a non-empty triangle in $G$ with inclusion-minimal interior among all non-empty triangles in $G$.
 Let $X \subset V$ be the (non-empty) set of vertices in the interior of $\Delta$, $G_{\rm out} = G - X$ be the triangulation induced by $V - X$, and $G_{\rm in}$ be the $4$-connected triangulation induced by $V(\Delta) \cup X$.
 By induction hypothesis there exists a $3$-interval representation $\{f(v) \mid v \in V - X\}$ of $G_{\rm out}$ of depth~$3$ satisfying the invariant \textbf{(I1)} and \textbf{(I2)}.
 In particular, the three vertices $x,y,z$ of $\Delta$, which now form an inner face of $G_{\rm out}$, are already assigned to up to three intervals each.
 We shall now extend this representation to the vertices in $X$.
  
 By invariant \textbf{(I2)} at least one edge of $\Delta$, say $xz$, is a displayed edge.
 Consider the vertices $u_x,u_y,u_z$ opposing $x,y,z$ in $G_{\rm in}$, respectively, and let $F_x$, $F_y$, $F_z$ be the decomposition of the inner edges of $G_{\rm in}$ given by Lemma~\ref{lem:inner-decomposition} (note that we can choose $x$, $y$, $z$ arbitrarily in Lemma~\ref{lem:inner-decomposition}).
 For convenience, if $G_{\rm in} \cong K_4$, let $F_z$ additionally have vertex $x$ as a one-vertex component.
 We define, based on this decomposition, three intervals for each vertex $v \in X$ as follows.
 
 \begin{itemize}
  \item Represent the path $F_x - \{x\}$ on an unused portion of the real line by using one interval per vertex, and in such a way that every vertex and every edge is displayed.
  Let this representation be denoted by $\{f_1(v) \mid v \in X\}$.
  
  \item Consider $y$ to be the root of the tree $F_y$, and for each vertex $v$ in $F_y - \{y\}$ consider the parent $w$ of $v$ in $F_y$, i.e., the neighbor of $v$ on the $v$-to-$y$ path in $F_y$.
  Create a new interval for $v$ strictly inside the displayed portion of $w$.
  If $w = y$, such a portion exists in $f(y)$ as \textbf{(I1)} holds for $f$, and if $w \neq y$, such a portion exists in $f_1(w)$.
  Let these new intervals be denoted by $\{f_2(v) \mid v \in X\}$.
  
  \item Consider the three trees that are the components of $F_z - \{u_yz\}$, that is, of $F_z$ after removing the edge $u_yz$.
  Say $T_1$ contains vertex $u_y$ (and possibly no other vertex), $T_2$ contains $z$, and $T_3$ contains $x$.
  Consider $u_y$, respectively $z$ and $x$, to be the root of $T_1$, respectively $T_2$ and $T_3$.
  For each vertex $v$ in $T_1 - \{u_y\}$, respectively $T_2 - \{z\}$ and $T_3 - \{x\}$, consider the parent $w$ of $v$ in $T_1$, respectively $T_2$ and $T_3$, and create a new interval for $v$ in the displayed portion of $w$.
  (Again, if $w \in \{x,z\}$, such a portion exists in $f(w)$ as \textbf{(I1)} holds for $f$, and if $w \notin \{x,z\}$, such a portion exists in $f_1(w)$.)
  Let these new intervals be denoted by $\{f_3(v) \mid v \in X - \{u_y\}\}$.
  
  \item Finally, create a new interval $f_3(u_y)$ in the displayed portion of edge $xz$ (to represent $xu_y$ and $zu_y$).
 \end{itemize}

 Defining $f(v) = f_1(v) \cup f_2(v) \cup f_3(v)$ for each $v \in X$, it is straightforward to check that $\{f(v) \mid v \in V\}$ is a $3$-interval representation of $G = G_{\rm out} \cup G_{\rm in}$ of depth~$3$.
 Moreover, the edge $xz$, every inner edge of $G_{\rm in}$ except for $u_yx$ and $u_yz$, and every vertex of $X$ is displayed. Hence this representation satisfies our invariants \textbf{(I1)} and \textbf{(I2)}, which concludes the proof.
\end{proof}

\section{The approach of Scheinerman and West}\label{sec:old}

To explain the proof strategy of Scheinerman and West, we introduce some more terminology from their paper.
For a plane embedded graph $G$ let $G^0$ denote the (outerplanar) subgraph induced by its \emph{external vertices}, i.e., those lying on the unbounded, outer face.
The edges of $G^0$ that bound the outer face are the \emph{external edges}, while those that bound two inner faces are the \emph{chords}.
The graph $G^0$ is considered together with its decomposition into its inclusion-maximal $2$-connected subgraphs, called blocks.
Furthermore, we fix for each component of $G^0$ a non-cut-vertex $z$ as the root of that component and say that $z = z_H$ is the root of its corresponding block $H$.
For each remaining block $H$ of $G^0$ let $z_H$ be the cut-vertex of $G^0$ in $H$ that is closest to the root of the corresponding component.
For a chord $xy$ of $G^0$ contained in block $H$ with root $z_H$ and a third vertex $u \neq x,y$ of $H$, say that $u$ is \emph{on the $z^*_H$-side} of $xy$ if $u$ and $z_H$ lie in different connected components of $H - \{x,y\}$. 
%

The argument of Scheinerman and West proceeds by induction on the number of vertices in $G$ with a stronger induction hypothesis on how $G^0$, i.e., the subgraph of $G$ induced by the external vertices, is represented.
Every edge $xy$ of $G^0$ shall be represented in such a way that for a possible future vertex $v$ which is adjacent to $x$ and $y$ we can define only one interval $I$ that has at least two of the following three properties:
\[
\textbf{P1: } I \cap f(x) \neq \emptyset, \qquad \textbf{P2: } I \cap f(y) \neq \emptyset, \qquad \textbf{P3: } I \text{ is displayed}.
\]
If $xy$ is displayed, it is easy to find an interval $I$ for $v$ having properties \textbf{P1} and \textbf{P2}.
Moreover if $x$, respectively $y$, has a broken end, it is easy to satisfy \textbf{P1}, respectively \textbf{P2}, and \textbf{P3}.
However if $xy$ is not displayed and neither $x$ nor $y$ has a broken end, Scheinerman and West propose to alter the existing representation of another vertex $u$ whose interval covers an endpoint of $x$.
As their modification, which is depicted in Figure~\ref{fig:reusable-end}, splits an interval of $u$ into two, it is necessary that $u$ appears at most twice so far, meaning that $f(u)$ consists of at most two intervals.
This shall be achieved by specifying the current representation of $G$ quite precisely.

\begin{figure}[tb]
	\centering
	\includegraphics[width=\textwidth]{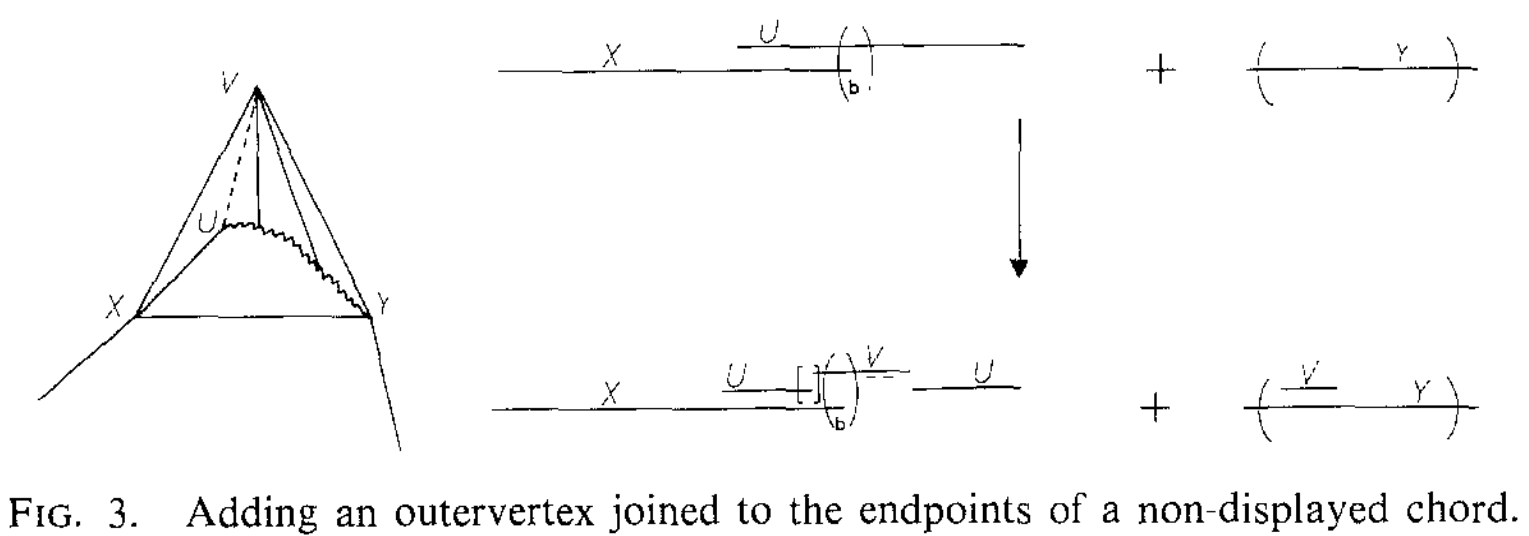}
	\caption{Figure~3 from~\cite{Sch-83} illustrating how to use a $z^*_H$-reusable endpoint $b$ of $f(x)$, which is covered by some vertex $u$ and assigned to a chord $xy$ of $H$, in order to insert a new vertex $v$ that is adjacent to $x$ and $y$.
		Note that the interval for $u$ is split into two.}
	\label{fig:reusable-end}
\end{figure}

Given a rooted plane embedded graph $G = (V,E)$, i.e., with a fixed root $z_H$ for each block $H$ of $G^0$, a representation $\{f(v) \mid v \in V\}$ is called \emph{P-special} if every vertex is displayed and each of the following holds:

\begin{enumerate}[label = (\arabic*)]
	\item Each root is represented by one interval and every other external vertex is represented by at most two intervals.
	\item For each block $H$ of $G^0$ all edges incident to $z_H$ are displayed.
	\item Each non-displayed edge $xy$ of any block $H$ of $G^0$ is assigned to an endpoint $b$ of $f(x)$ or $f(y)$, say $f(x)$, such that the following hold.
	\begin{enumerate}[label = (\arabic{enumi}.\arabic*)]
		\item If $xy$ is an external edge, the endpoint $b$ is a broken end.
		
		\item If $xy$ is a chord, the endpoint $b$ is a broken end, or $b$ is covered by $f(u)$ for only one other vertex $u$, where additionally $u$ is on the $z^*_H$-side of $xy$ and $ux$ is an \underline{external} edge of $G^0$.
		An endpoint $b$ satisfying this condition is called \emph{$z^*_H$-reusable} for edge $xy$.
		\label{enum:reusable}
		
		\item For each endpoint $b$ there is at most one edge assigned to $b$.
		
		\item Each vertex covers at most one endpoint which has an edge assigned.
	\end{enumerate}  
\end{enumerate}

Now Scheinerman and West propose to show by induction on the number of vertices in $G$ that every rooted plane embedded graph $G$ admits a P-special $3$-interval representation.
They remove a small set of carefully chosen external vertices from $G$, induct on the smaller instance to obtain a P-special representation, create intervals for the removed vertex/vertices, extend and/or alter the existing representation, and argue that the result is a P-special $3$-interval representation of~$G$.

The problem lies in the last step.
Some removed vertex $v$ may cover an edge $ux$ that is external in $G - v$, but a chord in $G$.
But it may be that some non-displayed chord $xy$ of $G - v$ is assigned to an endpoint $b$ of $f(x)$ that is $z^*_H$-reusable and covered by $u$.
This assignment cannot be kept, since $ux$ is no longer external, which however is required for $z^*_H$-reusability as it is defined in~\ref{enum:reusable} above. Thus, the invariants of the induction cannot be maintained.

\medskip

Let us explain in more detail below (discussing only the relevant cases) on basis of a small example graph that the above problem can indeed occur, and why some straightforward attempt to fix it does not work.
To this end, consider the plane embedded $9$-vertex graph $G$ in the top-left of Figure~\ref{fig:counterexample-decompose}.
The thick edges highlight the subgraph $G^0$, which has only one block $H$, and let us pick vertex $x_1$ to be the root of that block, i.e., $z_H = x_1$.
The base case of the inductive construction is an independent set of vertices.
Otherwise, we identify a set $X$ of external vertices that we want to remove as follows.
Consider a leaf-block $H$ of $G^0$, i.e., one that contains no root of another block, its root $z_H$, and distinguish the following cases.

\begin{figure}[tb]
	\centering
	\includegraphics{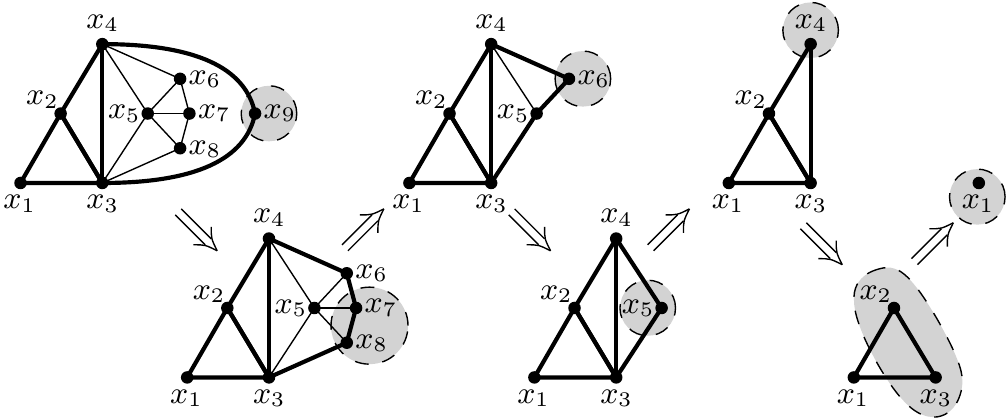}
	\caption{A plane embedded $9$-vertex graph $G$ with one block $H$ and root $z_H = x_1$, and how it is reduced to an independent set according to the induction rules in~\cite{Sch-83}.
		In each of the six steps, the set $X$ is highlighted and the smaller graph is obtained by removing all vertices in $X$.}
	\label{fig:counterexample-decompose}
\end{figure}

\begin{description}
	\item[Case~I: Every chord of $H$ is incident to $z_H$.]
	
	In this case label the vertices of $H$ as $z_H,v_1,\ldots,v_k$, $k \geq 1$, in this cyclic order around the outer face of $H$.
	If $k \geq 2$, let $X = \{v_1,\ldots,v_p\}$, where $p$ is the smallest integer greater than or equal to $2$ for which $v_p$ has degree~$2$ in $H$.
	
	\item[Case~II: Some chord of $H$ is not incident to $z_H$.]
	
	Consider an inner face $C$ of $H$ that does not contain $z_H$ and is bounded by only one chord $xy$, and distinguish further.
	
	\begin{description}
		\item[Case~IIa: $C$ is a triangle.]
		
		Let $X = \{v\}$, where $v \neq x,y$ is the third vertex in $C$.
		
		\item[Case~IIb: $C$ has at least four vertices.]
		
		In this case label the vertices of $C$ as $x,v_1,\ldots,v_k,y$, $k \geq 2$, in this cyclic order around $C$, and let $X = \{v_1,v_2\}$.
	\end{description}
\end{description}

Figure~\ref{fig:counterexample-decompose} shows how our example graph $G$ is reduced to an independent set according to these rules. 
Next, we consider these steps in reverse order and construct an interval representation following the case distinction as proposed in~\cite{Sch-83}.
Let us refer to Figure~\ref{fig:counterexample} for a step-by-step illustration of this construction.
The problem arises in Step~5.

\begin{description}
	\item[Step 1: base case.]
	
	Represent $x_1$ with a single interval.
	
	\item[Step 2: Case~I, $X = \{x_2,x_3\}$.]
	
	Insert intervals for $x_1,x_2$ into the displayed interval for $x_1 = z_H$.
	Place overlapping, displayed intervals for $x_1,x_2$ in an unused portion of the real line.
	
	\item[Step 3: Case~IIa, $X = \{x_4\}$.]
	
	As the chord $xy = x_2x_3$ is displayed, called subcase (1), add a displayed interval for $v = x_4$ and place the second interval available for $v$ in the displayed portion for $xy$.
	Assign one broken end of $v$ to each of the external edges $xv=x_2x_4$ and $yv=x_3x_4$.
	
	\item[Step 4: Case~IIa, $X = \{x_5\}$.]
	
	As the chord $xy = x_3x_4$ is not displayed, but assigned to a broken end $b$ of $x = x_4$, called subcase (2), add the displayed interval for $v = x_5$ so that it overlaps $f(x)$ at $b$.
	Add a second interval for $v$ in the displayed portion of $f(y)$.
	Assign the chord $xy$ to the $z^*_H$-reusable endpoint $b$, which is here covered by $u = x_5$ with $ux$ being an external edge.
	
	\item[Step 5: Case~IIa, $X = \{x_6\}$.]
	
	As the chord $xy = x_4x_5$ is displayed, we are in subcase (1) again.
	We add a displayed interval for $v = x_6$ and place a second interval for $v$ in the displayed portion for $xy$.
	Now the problem is that the chord $x_3x_4$ can no longer be assigned to the endpoint $b$, since the edge $x_4x_5$ is no longer external.
	
	One might be tempted to assume that the problem can be fixed by relaxing the definition of $z^*_H$-reusability as follows.
	Instead of $ux$ being external, maybe it suffices to require that $u$ appears at most twice so far.
	Let us call this new concept \emph{almost $z^*_H$-reusability}.
	However, the assumption that $ux$ is external is crucial for Case~IIb, as we will demonstrate in Steps 6 and 7 below.
	
	\begin{figure}[tbp]
		\centering
		\includegraphics{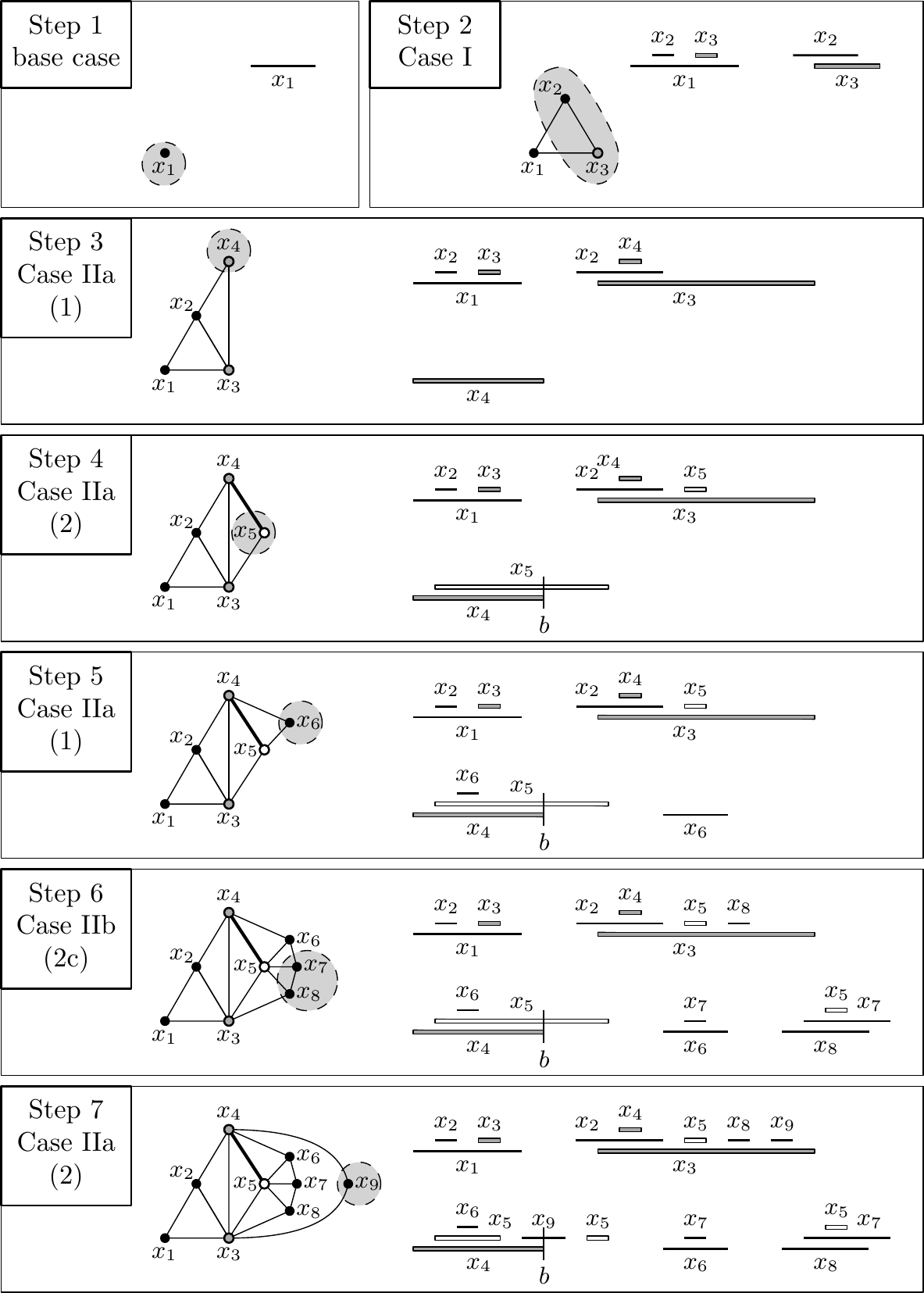}
		\caption{The steps of the inductive construction in~\cite{Sch-83} applied to the graph in Figure~\ref{fig:counterexample-decompose}.
			Note that after Step~5 the chord $x_3x_4$ is not displayed and the right endpoint of the longer interval for $x_4$ is not $z^*_H$-reusable since edge $x_4x_5$ is not external.
			Further note that after Step~7 vertex $x_5$ is represented by four intervals.}
		\label{fig:counterexample}
	\end{figure}
	
	\item[Step 6: Case~IIb, $X = \{x_7,x_8\}$.]
	
	Assign $v_1 = x_8$ an interval in the displayed interval for $x = x_3$, and assign $v_2 = x_7$ an interval in the displayed interval for $v_3 = x_6$.
	As the chord $xy = x_3x_4$ is not displayed, but assigned to an almost $z^*_H$-reusable endpoint $b$ of $y = x_4$, we are in subcase (2c) of~\cite{Sch-83}.
	Here it is concluded in~\cite{Sch-83} that the vertex $u$ covering $b$ must be $u = v_k = x_6$ due to the externality of edge $ux$.
	Hence it would suffice to place overlapping displayed intervals for $v_1$ and $v_2$ in an unused portion of the line and add intervals for the inner neighbors of $v_1,v_2$ in the displayed portion for $v_1$, $v_2$, and $v_1v_2$, according to the subset of these vertices that they are adjacent to.
	
	However, in our example we have $u = x_5$, which at least appears only twice so far, but which is an inner neighbor of $v_1$ and $v_2$, and thus has to spend its third interval in the displayed portion of $v_1v_2$.
	We end up with chord $xy$ still assigned to the endpoint $b$ that is covered by $u$, but neither is $ux$ external, nor is $u$ appearing only twice so far.
	In particular, we are not prepared to insert a future vertex adjacent to $x$ and $y$, as we illustrate in Step~7.
	
	\item[Step 7: Case~IIa, $X = \{x_9\}$.]
	
	The chord $xy = x_3x_4$ is not displayed, but assigned to the endpoint $b$ of $x = x_4$, which is covered by $u = x_5$ with $u$ however being internal and appearing three times already.
	Thus, if we apply the proposed modification as illustrated in Figure~\ref{fig:reusable-end}, we split one interval of $u$ into two, causing $f(u)$ to consist of four intervals.
\end{description}

This concludes our example.
Let us remark that the fourth interval for $x_5$ in Step~7 of Figure~\ref{fig:counterexample} might seem superfluous, but is actually needed to give $u$ a displayed portion.
This in turn is necessary as there might be a neighbor $w$ of $x_5$ of degree $1$.
For example, with $w$ inside face $x_5x_6x_7$, it would have been inserted between Steps~5 and 6, according to Case~I with $z_H = x_5$, spending one interval in the displayed portion of $u$.

Also note that in this particular example other ways of assigning endpoints in Step~5 would have allowed the process to continue.
However, the first author together with Daniel Gon\c{c}alves tried some time to obtain new invariants to fix this induction but did not succeed.


\section{Concluding remarks}\label{sec:conclusion}

As mentioned above, Axenovich~\textit{et al.}~\cite{Axe-13} used some steps of the construction from Section~\ref{sec:new} to prove that every graph $G$ whose edges decompose into $k-1$ forests and another forest of maximum degree~$2$ admits a $k$-interval representation of depth at most~$2$.
Note that if $G$ is a $4$-connected triangulation with outer triangle $\Delta_{\rm out}$, the decomposition of $G - E(\Delta_{\rm out})$ given by Lemma~\ref{lem:inner-decomposition} can be easily extended to a decomposition of all edges in $G$ into two forests and a path, which immediately gives the following.

\begin{theorem}\label{thm:4-connected-case}
	If $G$ is planar and $4$-connected, then $i_2(G) \leq 3$.
\end{theorem}




Given that the largest interval number among all planar graphs is~$3$, while the largest track number is~$4$, it remains open to determine whether the largest local track number among all planar graphs is~$3$ or~$4$, c.f.,~\cite[Question 19]{Kna-16}. 
We feel that the gluing of several triangulations along separating triangles is likely to be possible along the lines discussed in~\cite{Kna-16} for planar $3$-trees.
However, finding a $3$-local track representation of a $4$-connected planar triangulation, strengthening Theorem~\ref{thm:4-connected-case}, seems to be more difficult. What is the  largest local track number among all planar graphs?

\paragraph{Acknowledgments.}

We would like to thank Maria Axenovich and Daniel Gon\c{c}alves for fruitful discussions.
We also thank the anonymous reviewers for their comments and suggestions.
Moreover, we thank Ed Scheinerman and Douglas West for their helpful comments on an earlier version of this manuscript and for providing us a copy of Ed's PhD thesis.
The second author was partially supported by ANR projects ANR-16-CE40-0009-01 and ANR-17-CE40-0015 and by the Spanish Ministerio de Economía,
Industria y Competitividad, through grant RYC-2017-22701.

\bibliography{sw}

\begin{thebibliography}{10}

\bibitem{Axe-13}
M.~Axenovich, A.~Beveridge, J.~P. Hutchinson, and D.~B. West.
\newblock Visibility number of directed graphs.
\newblock {\em SIAM J. Discrete Math.}, 27(3):1429--1449, 2013.

\bibitem{Bal-04}
J.~Balogh, P.~Ochem, and A.~Pluh{\'a}r.
\newblock On the interval number of special graphs.
\newblock {\em J. Graph Theory}, 46(4):241--253, 2004.

\bibitem{Bal-99}
J.~Balogh and A.~Pluh{\'a}r.
\newblock A sharp edge bound on the interval number of a graph.
\newblock {\em J. Graph Theory}, 32(2):153--159, 1999.

\bibitem{Cha-09}
J.~Chalopin and D.~Gon{\c{c}}alves.
\newblock Every planar graph is the intersection graph of segments in the
  plane.
\newblock In {\em Proceedings of the forty-first annual ACM symposium on Theory
  of computing}, pages 631--638. ACM, 2009.

\bibitem{Cha-12}
S.~{Chaplick} and T.~{Ueckerdt}.
\newblock Planar graphs as {VPG}-graphs.
\newblock {\em Journal of Graph Algorithms and Applications}, 17(4):475--494,
  2013.

\bibitem{deQ-16}
A.~B. De~Queiroz, V.~Garnero, and P.~Ochem.
\newblock On interval representations of graphs.
\newblock {\em Discrete Appl. Math.}, 202:30--36, 2016.

\bibitem{Gon-07}
D.~Gon{\c{c}}alves.
\newblock Caterpillar arboricity of planar graphs.
\newblock {\em Discrete Math.}, 307(16):2112--2121, 2007.

\bibitem{Gon-09}
D.~Gon{\c{c}}alves and P.~Ochem.
\newblock On star and caterpillar arboricity.
\newblock {\em Discrete Math.}, 309(11):3694--3702, 2009.

\bibitem{Jia-13}
M.~Jiang.
\newblock Recognizing d-interval graphs and d-track interval graphs.
\newblock {\em Algorithmica}, 66(3):541--563, 2013.

\bibitem{Kna-16}
K.~{Knauer} and T.~{Ueckerdt}.
\newblock {Three ways to cover a graph.}
\newblock {\em {Discrete Math.}}, 339(2):745--758, 2016.

\bibitem{Kna-17}
K.~{Knauer} and T.~{Ueckerdt}.
\newblock {Decomposing $4$-connected planar triangulations into two trees and
  one path.}
\newblock {\em {J. Comb. Theory, Ser. B}}, 134:88--109, 2019.

\bibitem{Sch-84}
E.~R. Scheinerman.
\newblock {\em Intersection Classes and Multiple Intersection Parameters of
  Graphs}.
\newblock PhD thesis, Princeton University, 1984.

\bibitem{Sch-83}
E.~R. {Scheinerman} and D.~B. {West}.
\newblock {The interval number of a planar graph: Three intervals suffice.}
\newblock {\em {J. Comb. Theory, Ser. B}}, 35:224--239, 1983.

\bibitem{Stu-15}
P.~Stumpf.
\newblock On covering numbers of different kinds.
\newblock Bachelor thesis, Karls\-ruhe Institute of Technology, August 2015.

\bibitem{Tho-86}
C.~Thomassen.
\newblock Interval representations of planar graphs.
\newblock {\em {J. Comb. Theory, Ser. B}}, 40(1):9--20, 1986.

\bibitem{Tro-79}
W.~T. {Trotter} and F.~{Harary}.
\newblock {On double and multiple interval graphs.}
\newblock {\em {J. Graph Theory}}, 3:205--211, 1979.

\bibitem{Wes-84}
D.~B. {West} and D.~B. {Shmoys}.
\newblock {Recognizing graphs with fixed interval number is NP-complete.}
\newblock {\em {Discrete Appl. Math.}}, 8:295--305, 1984.

\end{thebibliography}
\bibliographystyle{abbrv}

\end{document}